\documentclass[a4paper,11pt]{article}
\usepackage{subcaption}
\usepackage{calc}
\usepackage{microtype}
\usepackage{amsmath}
\usepackage{amssymb}
\usepackage{amsthm}
\usepackage{color}
\usepackage{graphicx}
\usepackage{caption}
\usepackage{subcaption}
\usepackage{pst-pdf,pstricks-add}
\usepackage[hidelinks]{hyperref}
\usepackage{tikz}
\usetikzlibrary{plotmarks}
\usepackage{pgfplots}
\usepackage{dsfont}
\usepackage{listings}
\usepackage{comment}
\usepackage{cite}
\usepackage{mathrsfs,graphicx,color,float, indentfirst,textcomp}
\usepackage{setspace}
\usepackage{latexsym,amssymb,lscape,amsmath,amsthm,amsfonts,amssymb,cite,enumerate}
\usepackage{lscape}
\usepackage{multirow} 
\usepackage{xcolor}
\usepackage{geometry}
\usepackage[normalem]{ulem}
\usepackage{algorithm2e}
\SetKwComment{Comment}{/* }{ */}
\numberwithin{equation}{section}
\newtheorem{remark}{Remark}[section]

\newtheorem{lemma}{Lemma}[section]
\newtheorem{theorem}{Theorem}[section]

\newtheorem{proposition}{Proposition}[section]
\newtheorem{assumption}{Assumption}[section]

\newcommand{\Rd}{ {\mathbb{R}^\textup{d}}}
\newcommand{\R}{ \mathbb{R}}
\newcommand{\RR}{ \mathbb{R}}
\newcommand{\E}{\mathcal{E}}
\newcommand{\Par}{\mathscr{P}}

\newcommand{\eps}{\epsilon}
\newcommand{\ve}{\varepsilon}
\renewcommand{\d}{\textup{d}}
\newcommand{\supp}{\textup{supp}}
\newcommand{\law}{\textup{Law}}

\newcommand{\OX}{\overline{X}}
\def\argmin{\textup{argmin}}

\def\be{\begin{equation}}
\def\ee{\end{equation}}
\def\bea{\begin{eqnarray}}
\def\eea{\end{eqnarray}}

\title{Kinetic description and convergence analysis of genetic algorithms\\ for global optimization}
\date{\today}
	\author{Giacomo Borghi\thanks{RWTH Aachen University, Institut für Geometrie und Praktische Mathematik, Templergraben 55, 52062, Aachen, Germany ({\tt borghi@eddy.rwth-aachen.de}, corresponding author)}
	\and 
	Lorenzo Pareschi\thanks{Maxwell Institute for Mathematical Sciences and
Department of Mathematics, School of Mathematical and Computer Sciences (MACS), HWU Edinburgh, UK.} $^{,}$\thanks{Department of Mathematics and Computer Science, University of Ferrara, Italy.}}

\begin{document}
	\maketitle
\abstract{
Genetic Algorithms (GA) are a class of metaheuristic global optimization methods inspired by
the process of natural selection among individuals in a population. Despite their widespread use, a
comprehensive theoretical analysis of these methods remains challenging due to the complexity of the
heuristic mechanisms involved. In this work, relying on the tools of statistical physics, we take a first
step towards a mathematical understanding of GA by showing how their behavior for a large number
of individuals can be approximated through a time-discrete kinetic model. This allows us to prove
the convergence of the algorithm towards a global minimum under mild assumptions on the objective
function for a popular choice of selection mechanism. Furthermore, we derive a time-continuous
model of GA, represented by a Boltzmann-like partial differential equation, and establish relations
with other kinetic and mean-field dynamics in optimization. Numerical experiments support the
validity of the proposed kinetic approximation and investigate the asymptotic configurations of the
GA particle system for different selection mechanisms and benchmark problems.}

\bigskip
\bigskip


\bigskip
\bigskip




	\tableofcontents

\section{Introduction}
Many real-world problems are too complex to be solved using traditional optimization methods. Metaheuristic optimization provides a powerful set of tools for tackling these challenges, enabling researchers and practitioners to address a wide range of issues in areas such as engineering, finance, healthcare, and more. Unlike deterministic optimization methods, meta-heuristics do not guarantee the search for the global optimum, but provide a suitable solution to the application under consideration by exploring the search space through a random process guided by some heuristics. Among the most famous metaheuristic optimization algorithms for global optimization let us mention simulated annealing \cite{kirkpatrick1983annealing}, particle swarm optimization \cite{kennedy1995particle}, genetic algorithms \cite{Holland:1992:ANA:531075}, ant colony optimization \cite{dorigo96} and consensus-based optimization \cite{pinnau2017consensus}. We refer to \cite{talbi2009meta, glover2003} for a general introduction to the topic.  

From a statistical physics perspective, metaheuristic algorithms can be viewed as a form of stochastic optimization that samples the solution space and explores its properties using principles of statistical mechanics. The population of solutions can be thought of as a thermodynamic ensemble, where each individual solution corresponds to a microstate and the entire population represents a macrostate. The fitness function acts as an energy function that drives the evolution of the population towards higher fitness states, similar to how a physical system evolves towards a lower energy state.

Metaheuristic optimization and statistical physics share many similarities in their approaches to exploring large search spaces. Both fields deal with systems that involve a large number of components, such as solutions to a problem in optimization or particles in a physical system. In these systems, the search space is vast, and exploring it exhaustively is often infeasible. To navigate these large search spaces, both metaheuristic optimization and statistical physics use probabilistic methods that guide their search in a principled way. The use of probability allows both fields to handle complex interactions and dependencies between different components of the system.

Although metaheuristics have proven to be effective in solving complex optimization problems and have been successfully applied in many real-world applications, proving rigorous convergence to the global minimum has revealed to be challenging due to the stochastic nature of the algorithms, the lack of a rigorous mathematical foundation, and the large solution spaces. Drawing on the principles of statistical physics has led to significant advances in the development of more efficient and effective optimization algorithms. In particular the use of mathematical tools designed for mean-field and kinetic equations to describe the behavior of large systems of interacting particles has been adapted to these new fields, demonstrating the ability to prove convergence to the global minimum under fairly mild assumption on the fitness function. Recently, these ideas have led to a new view of metaheuristic optimization by considering the corresponding continuous dynamics described by appropriate kinetic equations of Boltzmann type \cite{benfenati2021binary,albi2023kinetic} and mean-field type \cite{pinnau2017consensus,Grassi21,carrillo2019consensus,carrillo2018analytical,grassi23,fornasier2021convergence,Bolte2023}, even in constrained and multi-objective contexts  \cite{borghi2023multi,borghi2023constrained,carrillo2021constrained,fhps20-1}.

As a result of the great success of meta-heuristic methods for global optimization problems, many programming languages have developed specific libraries that implement such algorithms, for example, MATLAB since the 1990s has built in three derivative-free optimization heuristic algorithms: simulated annealing, particle swarm optimization and genetic algorithm (see Table \ref{tab1}). While the first two approaches have recently been analyzed mathematically through appropriate mean-field reformulations \cite{Pavlio23, Mon18,Grassi21, Hui22, pareschi23}, to the authors' knowledge a similar mathematical theory based on genetic algorithms is still not available.  Preliminary has been done in \cite{albi2023kinetic} where authors proposed and analyzed a metaheuristic including typical mechanisms of genetic algorithms like survival-of-the-fittest and mutation strategies.


In this work, we propose a kinetic description of a paradigmatic genetic algorithm and theoretically
analyze its convergence properties and relations to other metaheuristics.  We do this by interpreting the algorithm run as a single realization of a Markov chain process describing the evolution of a system of $N$ particles. Then, we perform a kinetic approximation by assuming propagation of chaos for $N \gg 1$, allowing to describe the dynamics as a mono-particle system. We show that the modeling procedure is flexible enough to describe different variants of genetic algorithms: both where the parents selection is fitness-based and rank-based. 

Thanks to the resulting kinetic approximation, we gain a deeper understanding of genetic
algorithms from two distinct perspectives. First, we are able to establish convergence guarantees
when employing Boltzmann selection of parents. In particular, we demonstrate that the mean of
the particle system converges to an arbitrarily small neighborhood of a global minimizer, provided
that the algorithm parameters satisfy specific criteria and under mild assumptions on the objective
function. Secondly, the kinetic modeling enables us to establish novel relationships between genetic
algorithms and other metaheuristics. We investigate this aspect by formally deriving a time-continuous version of the kinetic model, which takes the form of a partial differential equation
of Boltzmann type. 
After discussing differences and similarities with other kinetic models in
optimization, notably the Kinetic Binary Optimization (KBO) method \cite{benfenati2021binary}, we show how to recover
the Consensus-Based Optimization (CBO) method \cite{pinnau2017consensus} by taking a suitable parameter scaling.

The formal derivation of the kinetic model is supported by numerical simulations of the particle systems with increasing population size. We note in particular that, for large $N$, the empirical probability measure associated with the particle system becomes deterministic which validates the propagation of chaos assumption.
Moreover, we numerically investigate the asymptotic steady states of the particle system for different genetic algorithms and benchmark problems. Finally, we test GA against benchmark problems in dimension $\d = 10$ with different scaling of the parameters and population sizes.

The rest of the paper is organized as follows. In Section \ref{sec:2} we present the fundamental mechanisms of genetic algorithm and derive the kinetic description. Section \ref{sec:3} is devoted to the derivation of the time-continuous kinetic model and its relations with other kinetic and mean-field models in optimization. The theoretical convergence analysis of genetic algorithm with Boltzmann selection is carried out in Section \ref{sec:analysis}, while the numerical experiments are illustrated in \ref{sec:num}. Final remarks and future research directions are discusses in Section \ref{sec:conclusions}.

\begin{table}[tb]
\caption{Main differences between some classical metaheuristic optimization algorithms}
\begin{center}
\begin{tabular}{l|l}
\hline\hline
{\bf Algorithm} & {\bf Feature}\\
\hline\hline
Simulated Annealing & Generates a single point at each iteration.\\
(SA)& The sequence of points approaches an optimal solution.\\
\hline
Genetic Algorithm & Generates a population of points at each iteration.\\
(GA)& The fittest evolve towards an optimal solution.\\
\hline
Particle Swarm & Generates a population of points at each iteration. \\
Optimization (PSO)& The population reach consensus towards an optimal solution.\\
\hline
\end{tabular}
\end{center}
\label{tab1}
\end{table}%

\section{A probabilistic description of genetic algorithms}
\label{sec:2}

In the sequel we consider the following optimization problem
\begin{equation}\label{typrob}
x^\star \in \underset{x\in \Rd}{\argmin}\, \E(x)\,,
\end{equation}
where $\E:\Rd \to \RR$ is a given continuous objective function, which we wish to minimize over the search space $\Rd$, $\d \in \mathbb{N}$. 

\subsection {Genetic algorithms and binary interactions}

Genetic algorithms are a type of optimization algorithm inspired by the process of natural selection.
They use a population of potential solutions to a problem and iteratively evolve such population to find better solutions.
The algorithm works by selecting the fittest individuals in the population and using them to create new individuals through recombination and mutation. Over time, the population evolves towards better solutions through a process of selection \cite{Holland:1992:ANA:531075}. As in most metaheuristics, the central theme in genetic algorithms is robustness, namely a balance between efficiency and efficacy necessary to provide a reasonable solution even in the most critical situations.

In their original setting, genetic algorithms are applied to vectors of integers or vectors of binary numbers, characterizing the genes of the individuals. Here we describe the method in the case of real valued vectors as discussed for example in \cite{WRIGHT1991205}.
 
The algorithms iteratively update a population of trial points $x_i \in\Rd$, $i=1,\ldots,N$ and can be summarized as follows. 
\begin{enumerate}
\item Select a group of individuals in the current population, called parents, who contribute with their genes—the entries of their vectors—to their children. The algorithm selects individuals that have better fitness values as parents. This procedure is called \emph{selection}.

\item Given a pair of parents, the genetic algorithm adopts two main evolutionary dynamics to generate children:
\begin{itemize}
\item 
\emph{Crossover}: by combining the vectors of a pair of parents in different ways.
\item
\emph{Mutation}: by introducing random changes, or mutations, to a single parent.
\end{itemize}
\item The new generation is them composed by the generated children and part of the old generation. The individuals which are directly passed to the next generation can be either chosen randomly or following a fitness-based mechanism. In this latter case they are called \emph{elite} group.

\end{enumerate}

Let us describe the algorithms more in details. Let $\odot$ denote the component-wise (Hadamard) product between vectors of same length.
Assuming a pair of parents $(x, x_*)$ have been selected from the population $\{x_i \}_{i=1}^N$ accordingly to the selection procedure, we can describe the evolution process leading to an offspring  $x'$ as 
\be \label{eq:coll}
\begin{split}
x'& = \underbrace{(1-\gamma)\odot x + \gamma \odot x_*}_{\hbox{\footnotesize \emph{crossover}}} + \underbrace{\sigma D \odot \xi }_{\hbox{\footnotesize \emph{mutation}}}\,\\
&=: \mathcal{C}(x,x_*, \xi) 
\end{split}
\ee 
where  $\gamma \in [0,1]^\d$ is a crossover vector and $\xi$ a random vector with zero mean and unit variance that defines the random mutation with coefficient $\sigma>0$. 
For example, if all elements of $\gamma$ are either $0$ or $1$ we have a switch of components between parents, whereas if all elements are equal we have a uniform blending among parents. Other choices are also possible, even involving random choices of the crossover elements \cite{goldberg1989}. 
The mutation vector $D\in \Rd$ is assumed time dependent, so that as time increases (proportional to the generation number), the mutated solutions are generated closer to original solutions. When all components are equal we have isotropic scattering, otherwise we will have anisotropic noise. The fact that diffusion decreases over time introduces an effect similar to that of simulated annealing and is essential to avoid introducing mutations near the global minimum.

By assuming the number $N$ of individuals is the same for all generations, the generating procedure described above is then repeated $M\leq N$ times to generate the new individuals $x_i'$ with $i = 1, \dots, M$. The remaining $M-N$ individuals are directly taken from the previous generation $\{x_i \}_{i=1}^N$, eventually according to the elitist strategy. We note that $x_i$ is not necessary a parent of $x_i'$, and that, moreover, there is no particular relation between the two individuals. We point out that these algorithmic procedures are also known in the literature as \textit{evolution strategies} \cite{Diouane2015}.


The process just described presents several analogies with opinion and wealth dynamics described by microscopic interactions, where crossover correspond to binary wealth and opinion exchanges between agents and mutations to the random behavior of agents \cite{pareschi13}. In this respect, genetic algorithms have much in common with what can be called a binary genetic exchange process. The main difference is that the overall population is not conserved since the binary reproduction process is incremental (children do not replace parents as in classical kinetic theory) and it is only by the selective removal of individuals that the population can be contained.

\subsection{Selection as a sampling procedure}

The selection of individuals for the creation of new ones is a central step in genetic algorithms. The most common techniques are given by Roulette Wheel, Rank, Tournament, and Boltzmann Selection \cite{katoch2021,khalid2013selection}. In this section, we show how selection strategies can be interpreted as a sampling procedure of probability measures. This will be an essential ingredient to model the individuals' evolution in genetic algorithms as a mono-particle process. 

In the following, we consider a generic generation $\{ X_i\}_{i=1}^N$, representing positions in the search
space, and study how a couple of parents $(X, X_*)$ is selected. We denote with $f^N$ the empirical probability measure associated to the individuals, that is, 
\begin{equation}
f^N = \frac1N\sum_{i=1}^N \delta_{X_i}\,,
\end{equation}
where $\delta_{X_i}$ is the Dirac delta function.  For simplicity, we assume that each couple of parents $(X, X_*)$ is sampled independently from one another. 
 
The main distinction among selection techniques is whether they are fitness-based or rank-based. In a fitness-based mechanism, the probability of two parents to be picked directly depends on their objective values. In a rank-based procedure, instead, the individuals are first ranked according to their objective values and then the parents' probability depends on the rank only. 

\medskip 

\textbf{Fitness-based selection.}
Random Wheel Selection (RWS) is popular fitness-based selection strategy \cite{katoch2021}.  Let $g: \R \to (0,\infty)$ be a non-increasing continuous function which associates to any objective value $\E(x)$ its corresponding fitness $g(\E(x))>0$. The fitness function should be non-increasing as we aim to minimize $\E$, and so the fittest individuals are those with low-objective value. In RWS, the probability of individual $X_i$ to be picked as a parent $X$ is directly proportional to its fitness  $g(\E(X_i))$ 
\be
\label{eq:PRWS}
\mathbb{P}(X^{RWS} = X_i)  = \frac{g(\E(X_i))}{ \sum_{j =1}^N g(\E(X_j))}\, ,
\ee
and the two parents $X^{RWS}, X^{RWS}_*$ selected with RWS are chosen independently from one another. 

A common fitness choice is given by the exponential $g(\cdot) = \exp(-\alpha \E(\cdot))$ for some $\alpha>0$ leading to the so-called Boltzmann Selection (BS) strategy. The name comes from the fact that, with this choice of fitness, the parents turn out to be chosen according to the Boltzmann-Gibbs distribution associated with the objective $\E$ and inverse temperature $\alpha>0$:
\be
\label{eq:PBS}
\mathbb{P}(X^{BS} = X_i)  = \frac{e^{-\alpha \E(X_i)}}{ \sum_{j=1}^N e^{-\alpha \E(X_j)}} \, ,
\ee
for all $i = 1, \dots, N$ and the same holds for the second parent $X_*^{BS}$.

Fitness-based selection strategies can be generalized to arbitrary probability measures in a straightforward manner. Let $f \in\mathcal{P}(\Rd)$ be an arbitrary population distribution. We define the parents' probability measure $\Par_{RWS}[f] \in\mathcal{P}(\Rd \times \Rd)$ for RWS  as
\be
d\Par_{RWS}[f](x, x_*) :=  \frac{g(\E(x)) g(\E(x_*))}{\int g(\E(x))df(x) \int g(\E(x_*))df(x_*)} df(x) df(x_*)\,,
\ee
which, in case of BS, takes the particular form of
\be
d\Par_{BS}[f](x, x_*) :=  \frac{ e^{-\alpha (\E(x) +  \E(x_*))}}{\int e^{-\alpha\E(x)}df(x) \int e^{-\alpha\E(x_*)}df(x_*)} df(x) df(x_*)\,.
\label{eq:BS}
\ee
The above definitions generalize \eqref{eq:PRWS} and \eqref{eq:PBS} as direct computations lead to 
\[\law(X^{RWS}, X^{RWS}_*) = \Par_{RWS}[f^N] 
\quad \textup{and} \quad 
\law(X^{BS}, X^{BS}_*) = \Par_{BS}[f^N] 
\] 
whenever we consider empirical distributions.

\medskip

\textbf{Rank-based selection.}
We can perform the same generalization procedure for rank-based mechanism such as Rank Selection (RS). In RS, the probability of $X_i$ to be picked as a parent is dictated by its rank. Let $\textup{rank}(i) = j$ if $X_i$ is the $j$-th worst individual among the ensemble (assuming there are no individuals with same objective value). The RS parents' $X^{RS},X_*^{RS}$ probabilities are independently distributed and both given by 
\be
\label{eq:PRS}
\mathbb{P}(X^{RS} = X_i)  = 2\frac{\textup{rank}(i)}{N(N+1)} .
\ee
We note that the function $\textup{rank}$ can be equivalently defined as
\[ \textup{rank}(i) = \# \left \{x\;:\: \E(X_i) \leq \E(x)\right \} = N f^N \left(\{y\,:\, \E(X_i) \leq \E(y) \}\right)   \,.     \] 
According to the above definition, two or more different particles may also have the same rank if their fitness value coincide.  This leads to a reformulation of \eqref{eq:PRS} as
\be \notag
\mathbb{P}(X^{RS} = X_i)  = \frac{f^N \left(\{y\,:\, \E(X_i) \leq \E(y) \}\right) }{\sum_{j=1}^N f^N \left(\{y\,:\, \E(X_j) \leq \E(y) \}\right) } \, .
\ee
With such rewriting, it becomes clearer how to generalize rank-based selection to arbitrary probability measures $f\in \mathcal{P}(\Rd)$ of individuals. For RS, indeed, the parents' distribution is given by
\be\label{eq:RSpar}
d\Par_{RS}[f](x, x_*) :=  \frac{f\left(\{y\,:\, \E(x) \leq \E(y) \}\right) f\left(\{y\,:\, \E(x_*) \leq \E(y) \}\right)}{\left (\int f\left(\{y\,:\, \E(x) \leq \E(y) \}\right) df(x)\right)^2} df(x) df(x_*)\,.
\ee
Particle system with rank-based interactions and their mean-field limit have been investigated in \cite{banner2005atlas,jourdain2008propagation,jourdain2013chaos,kolli2018spde,bencheikh2022weak}. In the financial literature these systems are called rank-based models \cite[Chapter IV]{ben2009portfolio}, and the drift of the dynamics typically depends on the cumulative distribution function (cdf) of the stochastic process. Differently form rank-based models in the literature, in \eqref{eq:RSpar} the cdf is taken with respect to the objective function, and it is further normalized to obtain a probability measure.

We not that in all above cases, the parents' probability distribution $\Par[f]$ is absolutely continuous with respect to $f \otimes f$. In the following, we denote with $R[f]$ its density:
\begin{equation}
\label{eq:density}
d\Par[f](x,x_*) = R[f](x,x_*)\, df(x) df(x_*)\,,
\end{equation}
which can be seen as a reproduction kernel.

\subsection{Probabilistic description and kinetic approximation}

In this section, we show that genetic algorithms can be interpreted as single-realization of a system of $N$ Markov chains and derive a kinetic approximation of the system.
Let $\{X_i^{(k)}\}_{i=1}^N$ be $N$ random variables taking values in $\Rd$ and representing the population of individuals at the algorithmic step $k$. We denote the empirical probability measure of the generation as 
\begin{equation}
f^N_{(k)}=\frac1N \sum_{i=1}^N \delta_{X_i^{(k)}} \in \mathcal{P}(\mathcal{P}(\Rd))\,.
\end{equation}
We recall that in the genetic algorithm, part of the new generation is directly taken from the previous one (the \emph{elite} group), while the remaining part is given by newly born individuals. We indicate the fraction of new individuals with respect to the total population with $\tau \in (0,1]$. Assume we want to generate an individual $X^{(k+1)}_i$ of the generation $k+1$. To decide whether it belongs to the elite or not, we introduce a random variable $T_i^{(k)}$ with Bernoulli distribution $\textup{Bern}(\tau)$,  that is, $\mathbb{P}(T_i^{(k)} = 1) = \tau$ and $\mathbb{P}(T_i^{(k)} = 0) = 1 - \tau$. 
If $0$ is sampled, $X^{(k+1)}_i$ belongs to the elite group and coincide with a randomly sampled individual from the previous distribution, which we assume to be $X_{i}^{(k)}$ for simplicity. If $1$ is sampled, $X^{(k+1)}_i$ corresponds to a newly generated individual with parents sampled form $\Par[f^N_{(k)}]$.

Recall that the crossover and mutation procedure is determined by the map $\mathcal{C}$ defined in \eqref{eq:coll}. 
The individuals of the new generation are iteratively defined as  
\begin{equation}
 X^{(k+1)}_i = (1 - T^{(k)}_i)  X^{(k)}_i + T^{(k)}_i\, (X_i^{(k)})' \qquad \textup{for}\;\; i=1, \dots,N\,,
\label{eq:Nevolution}
\end{equation}
with $ (X_i^{(k)})' = \mathcal{C}(\tilde X_i^{(k)}, \tilde X_{i,*}^{(k)},\xi_i^{(k)})$ being the post-collisional position with parents $(\tilde X_{i}^{(k)}, \tilde X_{i,*}^{(k)})$ and mutation $\xi_i^{(k)}$ satisfying
\begin{equation} \notag
\law(\tilde X_i^{(k)}, \tilde X_{i,*}^{(k)}) = \Par[f^N_{(k)}], \qquad \textup{and} \qquad \law(\xi_i^{(k)}) =\mu
\end{equation}
with $\mu \in \mathcal{P}(\Rd)$ having zero mean and variance $I_\d$.

To study the large-time behavior of the individuals' interaction, we derive a kinetic approximation the system \eqref{eq:Nevolution}. A common strategy consists of making the so-called \textit{propagation of chaos} on the marginals \cite{sznitman1991chaos}. We assume that for large populations $N\gg1$, the individuals become independently identically distributed according to a distribution $f_{(k)}\in \mathcal{P}(\Rd)$ at all iterations $k\in \mathbb{Z}_+$. In particular, the population's probability measure as well as the parents' measure are then approximated as
\[	f^N_{(k)} \approx f_{(k)} \qquad \textup{and} \qquad \Par[f^N_{(k)}] \approx \Par [f_{(k)}]\,,
\]
respectively. 
Similarly as before, let $\overline T^{(k)} \sim \textup{Bern}(\tau)$ and $\overline{\xi}^{(k)} \sim \mu$.
Under the propagation of chaos assumption, update rule \eqref{eq:Nevolution} becomes independent of the index $i$ and can be re-written as the mono-particle process
\begin{equation}
\OX^{(k+1)} = (1 - \overline T^{(k)})  \OX^{(k)} + \overline T^{(k)}\, (\OX^{(k)})'
\label{eq:mono}
\end{equation}
with $\OX^{(k)}=\mathcal{C}(\tilde \OX^{(k)}, \tilde\OX_*^{(k)}, \overline \xi^{(k)})$ and
\begin{equation} \notag
\law(\tilde \OX^{(k)}, \tilde \OX_{*}^{(k)}) = \Par[f_{(k)}]\,. 
\end{equation}

Starting from an initial generation $f_{(0)}$, we have that for any measurable test function $\phi$ $f_{(k)}=\law(\OX^{(k)})$ satisfies
\begin{align} \label{eq:weak}
\int \phi (x)\,d f_{(k+1)}(x) &= (1 - \tau)\int \phi (x) \,df_{(k)}(x) + \tau \iint \phi(x')\,
d\Par[f_{(k)}](x,x_*)\,d\mu(\xi)
\end{align}
where we used the short notation $x'$ for $\mathcal{C}(x, x_*, \xi)$ introduced in \eqref{eq:coll}.

\begin{remark}
As we will see in the next section, equation \eqref{eq:weak} can be interpreted as an explicit discretization of a PDE of Boltzmann type. As a direct consequence, the Markov processes \eqref{eq:Nevolution} (and so the Genetic Algorithm itself) can be understood as a Monte Carlo simulation scheme of the kinetic model \cite{pareschi13}.
The approximation error introduced by the proposed kinetic approximation, therefore, is linked to the convergence properties of such stochastic particle-based numerical methods.

We refer to \cite{wagner2005stochastic} and the references therein for (non-quantitative) convergence results in the context of classical kinetic theory. We leave the convergence analysis of the particle scheme \eqref{eq:Nevolution} to the discrete Boltzmann equation \eqref{eq:weak} to future work.
\end{remark}

\section{A Boltzmann-type description}
\label{sec:3}

In this section, we show how to derive a continuous-in-time kinetic model which describes the individuals' evolution in genetic algorithms. We also show the relation between the GA model with other kinetic models in global optimization studied in the literature, namely those derived for Kinetic-Based Optimization (KBO) \cite{benfenati2021binary} and Consensus-Based Optimization (CBO) \cite{pinnau2017consensus} algorithms.

Recall $\tau$ quantifies the fraction of newborns per algorithmic step and therefore can be interpreted as the interaction frequency between the individuals. To derive a time-continuous description of the process we consider the limit $\tau\to 0$. 
Starting from \eqref{eq:weak}, we rearrange the terms to obtain that the sequence $(f_{(k)})_{k \in \mathbb{Z}_+}$ satisfies
\begin{align*}
\frac{\int \phi (x) \, df_{(k+1)}(x) - \int \phi (x)\,d f_{(k)}(x)}{\tau} &= \iint \phi(x')\,
d\Par[f_{(k)}](x,x_*)\, d\mu(\xi) -  \int \phi (x) \,d f_{(k)}(x) \\
& =  \iint \left(R[f_{(k)}](x,x_*) \phi(x') - \phi(x)   \right)\, df_{(k)}(x)\, df_{(k)}(x_*)\,. 
\end{align*}
By formally taking the limit $\tau \to 0$, the difference equation reduces to a Boltzmann-like equation of the form 
\be
\frac{d}{dt} \int \phi (x)\,df_{t}(x) =  \iiint  \left(R[f_t](x,x_*)\phi(x') - \phi(x)\right)\,df_t(x)\,df_t(x_*)\, d\mu(\xi)
\label{eq:boltzmann}
\ee
where $f_t, t>0,$ is a suitable limiting sequence for $f_{(k)}$. We refer the interested reader to \cite{cercignani1994mathematical,villani2002topicscollisional,pareschi13} for more details on kinetic theory and on the Boltzmann equation. 

In \eqref{eq:boltzmann}, $R[f](x,x_*)$ encodes the probability of $(x,x_*)$ to be picked as parents and $x' = \mathcal{C}(x,x_*, \xi)$ is the offspring generated from the crossover between $(x,x_*)$ and mutation $\xi$. With the test function $\phi(x) = 1$, we obtain that the right-hand-side of \eqref{eq:boltzmann} is zero, since $R[f_t]$ is a probability density with respect to $f_t \otimes f_t$. This means that, as expected, the total mass is preserved. This is the common approach in most implementations of genetic algorithms, since it avoids excessive population growth or oscillatory population behavior. We remark that kinetic modeling of non-conservative dynamics is also possible. In order to keep the mathematical analysis of the proposed model simpler, we leave this case as a future research direction.

\subsection{Relations to KBO}


When designing stochastic particle methods for optimization that are based on binary interactions, the objective function may enter the dynamics in two different ways: either in choosing which couple of particles interact, or in how they interact. GA employs the first strategy as the parents are chosen according to their fitness values. We recall in the following the Kinetic Binary Optimization (KBO) method \cite{benfenati2021binary}, which employs the latter strategy to underline differences and similarities between these two paradigms in metaheuristic optimization.


Similarly to GA, KBO can be derived as a Monte Carlo approximation of a kinetic equation of Boltzmann type.
Let $\lambda\in (0,1), \sigma>0$ be algorithm's parameters and $\xi,\xi_* \sim \mu= \mathcal{N}(0, I_\d)$. Given two particles $(x,x_*)$, the collisional dynamics defining the KBO system evolution is given by
\be
\begin{split}
x'&= \left(1 - \lambda \gamma^\E(x,x_*)\right)x + \lambda \gamma^\E(x,x_*) x_* + \sigma D(x,x_*) \odot \xi\\
x_*'&= \left(1 - \lambda \gamma^\E(x_*,x)\right)x_* + \lambda \gamma^\E(x_*,x) x + \sigma D(x_*,x)\odot \xi_* 
\end{split} \label{eq:collKBO}
\ee
with
\be
 \gamma^\E(x,x_*) = \frac{e^{-\alpha \E(x_*)}}{e^{-\alpha \E(x)} + e^{-\alpha \E(x_*)}} 
\label{eq:gammaKBO}
\ee
for some parameter $\alpha>0$. The $D(x,x_*)\in \Rd$ determines the strength of the random component and may depend on $x,x_*$ as well as on their objective values.

The binary interaction \eqref{eq:collKBO} shares exactly the same structure as the reproductive dynamics \eqref{eq:coll} of GA. In particular, every new offspring $x'$ is made of a convex combination of previous particles $(x,x_*)$ (the \textit{parents} in the genetic language) and it is randomly perturbed (\textit{mutation}) through a normally distributed random variable $\xi$.
As previously mentioned, the major difference between the two kinetic dynamics lies in where the objective function $\E$ enters.

In KBO, the particles objective values $\E(x), \E(x_*)$ are used through the weights \eqref{eq:gammaKBO} to determine the convex combination between $x,x_*$. In particular, say $x$ attains a larger objective value than $x_*$, then the new particle $x'$ is closer to $x_*$ than $x_*'$ is to $x$. Also, in KBO, the selection of parents $(x,x_*)$ is made uniformly over the particle distribution and does not depend on the objective $\E$. In kinetic theory, this is the case of Boltzmann models with Maxwell interaction between the particles \cite{pareschi13}, which leads to a kernel-free PDE of type
\be
\frac{d}{dt} \int \phi (x)\,df_{t}(x) =  \iiint  \left(\phi(x') - \phi(x)\right)\,df_t(x)\,df_t(x_*)\, d\mu(\xi)
\label{eq:boltzmannKBO}
\ee
The kinetic model \eqref{eq:boltzmann} for GA, instead, is non-Maxwellian due to the kernel $R[f]$ resulting from the parents' selection procedure. 

Under a suitable scaling of the parameters, authors in \cite{benfenati2021binary}  shows that  \eqref{eq:boltzmannKBO} reduces to a mean-field equation similar to the one derived for CBO methods. Under suitable assumptions, we show in the next section that the GA kinetic model \eqref{eq:boltzmann} also leads to a CBO-type equation, establishing an equivalence between KBO and GA in the mean-field regime.

\begin{remark}
$\,$

\begin{itemize}
\item We underline that in \cite{benfenati2021binary} a variant including also a macroscopic best estimate was also proposed.
\item Recently, authors in  \cite{albi2023kinetic} have proposed the Genetic Kinetic Based Optimization (GKBO) algorithm which combines ideas both from GA and KBO. By dividing the particles in two populations, the ``parents'' and the ``children'', in GKBO the objective function enters both in the interaction mechanism (as in KBO) and in how children become parents, and parents children (similarly to GA).
\end{itemize}
\end{remark}

\subsection{Fokker-Planck asymptotics and relation to CBO}
\label{sec:CBO}

We derive now the mean-field dynamics corresponding to the GA kinetic model \eqref{eq:boltzmann} via  the so called quasi-invariant scaling \cite{pareschi13}
\be
t \to t/\varepsilon,\quad \gamma \to \varepsilon \gamma,\quad \sigma^2 \to \varepsilon \sigma^2\,,
\label{eq:scaling}
\ee
where $\ve>0$ is a small parameter. The scaling introduced above corresponds to a dynamic in which
crossovers and mutations decrease in intensity but increase in number. 

By Taylor expansion for $\phi \in C_{0}^\infty(\Rd)$ we have
\[
\phi(x')=\phi(x) + (x'-x)\cdot \nabla_x \phi(x) + \frac12(x'-x)\cdot H[\phi](x)(x'-x)+O((x'-x)^3)
\]
where $H[\phi]$ is the Hessian matrix. We note that after scaling \eqref{eq:scaling}, it holds
\begin{equation}
x' - x = \ve \gamma \odot (x_* - x) + \sqrt{\ve}\sigma D \odot \xi\,.
\end{equation}
From the weak form \eqref{eq:boltzmann} and by Fubini's theorem, we get up to $O((x'-x)^3)$
\begin{multline*}
\frac{d}{dt} \int \phi (x)\,d f_{t}(x)  = \frac1{\ve} \iint \left( R[f_t](x,x_*) - 1 \right) \phi(x)\,df_t(dx)\,df_t(x_*)\\
+   \iint R[f_t](x,x_*) \gamma \odot (x_* - x)\cdot \nabla\phi(x)  \,df_t(x)\,df_t(x_*)\\
+  \frac12\iiint R[f_t](x,x_*)  \sigma^2 (D \odot \xi) \cdot  H[\phi](x) (D \odot \xi)  \,df_t(x)\,df_t(x_*)\, d\mu(\xi) + O(\ve)\,.
\end{multline*}
To be able to formally pass to the limit $\ve \to 0$, we must assume that for any test function $\phi $ it holds
\begin{equation*}
\int \left(R[f](x,x_*) - 1 \right) \phi(x) \,df(x)\, df(x_*) =1\,,
\end{equation*}
which can be achieved only by assuming that the reproduction kernel density is independent on $x$ 
\[R[f](x,x_*) = R[f](x_*)\,.\] 
This is somewhat contradictory with the original spirit of genetic algorithms, as it means that, out of the two parents $(x,x_*)$, only one of them, $x_*$, is chosen according to a selection mechanism, while the other, $x$, is chosen randomly among ensemble. In practice, such a reproduction kernel may mitigate the problem of premature convergence due to the excessive dominance of the fittest individuals \cite{goldberg1989}.

By passing to the limit $\ve \to 0$ and reverting to the strong form we get the mean-field model 
\be
\frac{\partial f_t}{\partial t}(x) 
+ \nabla_x \cdot \left(\mathcal{R}[f_t](x) f_t(x)\right) = \frac{\sigma^2}{2}\sum_{\ell=1}^\d \partial_{x_\ell x_\ell} (\mathcal{D}_\ell[f_t](x)f_t(x))
\label{eq:mf1}
\ee
with
\be
\mathcal{R}[f_t](x) =  \int R[f_t](x_*)\gamma\odot(x_*-x)\,df_t(x_*)
\qquad
\textup{and}
\qquad
\mathcal{D}_\ell[f_t](x) = D_\ell^2   \,.
\label{eq:mf2}
\ee
If we now assume that the reproduction function has a particular shape, we may recover the Consensus-Based Optimization (CBO) method \cite{pinnau2017consensus}. 
Consider the parents' probability distribution where the first parent is chosen randomly among the generation and the second one is chosen with respect to the Boltzmann-Gibbs probability distribution. This is given by the choice
\begin{equation}
d\Par_{CBO}[f_t](x,x_*) =  \frac{ e^{-\alpha  \E(x_*)}}{ \int e^{-\alpha\E(x_*)}df_t(x_*)} df_t(x) df_t(x_*)
\label{eq:PCBO}
\end{equation}
with corresponding kernel density 
\begin{equation*}
R_{CBO}[f_t](x_*) = \frac{ e^{-\alpha  \E(x_*)}}{ \int e^{-\alpha\E(x_*)}df_t(x_*)}\,.
\end{equation*}

Further assuming $\gamma = (\lambda, \dots, \lambda)^\top, \lambda >0$, we get 
\be
\mathcal{R}_{CBO}[f_t](x) = \lambda \int \frac{ e^{-\alpha  \E(x_*)}}{ \int e^{-\alpha\E(x_*)}df_t(x_*)}(x_* - x)\,df_t(x_*) = \lambda (m^\alpha[f_t] -x)
\ee
where, for any $f\in\mathcal{P}(\Rd)$, $m^\alpha[f]$ is the weighted mean given by
\be
\label{eq:wmean}
m^\alpha[f]:= \frac{\int x e^{-\alpha \E(x)}df(x)}{ \int e^{-\alpha\E(x)}df(x)}\,.
\ee
The resulting mean-field equation models a CBO-type evolution with Non-Degenerate diffusion (CBO-ND) and is given by
\be
\frac{\partial f_t}{\partial t}(x) 
+ \lambda \nabla_x\cdot \left((m^\alpha[f_t]-x) f_t(x)\right) = \frac{\sigma^2}{2} \sum_{\ell=1}^\d \partial_{x_\ell x_\ell}(D_\ell^2 f_t(x))\,.
\label{eq:mfCBO}
\ee

The original isotropic CBO dynamics \cite{pinnau2017consensus} can be obtained by, instead, taking a diffusion vector which depends on the overall distribution and the individual $x$, $D_{I} = D_{I}(f_t,x)$, of the particular form 
\be
D_{I,\ell}(f_t,x) = |m^\alpha[f_t] - x|\qquad \textup{for} \quad \ell=1,\dots, \d\,,
\label{eq:CBOiso}
\ee
while the anisotropic CBO diffusion \cite{carrillo2019consensus} is obtained by $D_{A} = D_{A}(f_t,x)$
\be
D_{A,\ell}(f_t,x) = (m^\alpha[f_t] - x)_\ell\qquad \textup{for} \quad \ell=1,\dots, \d\,.
\label{eq:CBOaniso}
\ee
Both choices correspond to an increase in mutations in the less fit individuals. This contradicts the elitist spirit of genetic algorithms where only the fittest survive, but it is in good agreement with the social spirit of CBO, where all points must be brought to the global minimum.

\begin{remark}
$\,$

\begin{itemize}
\item 
Despite the non-degenerate diffusion, deriving explicit steady-state solutions to \eqref{eq:mf1}-\eqref{eq:mf2} is a non-trivial task due to the nonlinear reproduction kernel $R[f_t]$. When the diffusive behavior is degenerate and linked to the reproduction kernel, as in the case of the CBO mean-field dynamics \eqref{eq:mfCBO} with \eqref{eq:CBOiso}, steady state solutions are given by any Dirac measure. We note that Gaussian steady states centered around $m^\alpha[f_t]$ can be obtained by considering the Consensus-Based Sampling dynamics suggested in \cite{carrillo2022sampling}. In Section \ref{sec:num} we will numerically investigate whether the GA particle system asymptotically converges to steady states.
\item In CBO methods, anisotropic exploration \eqref{eq:CBOaniso} has proven to be more suitable for high dimensional problems \cite{carrillo2019consensus}, and the same has been observed for KBO \cite{benfenati2021binary}. The same idea can be applied to GA by taking $D = D(x,x_*) = x_* - x$, with the consequence that the offspring's mutation strength depends on the distance between the two parents. The corresponding mean-field model takes again the form of \eqref{eq:mf1}, but with now 
\be
\mathcal{D}_\ell [f_t](x) = \int R[f_t](x_*) (x_*-x)^2_\ell \,df_t(x_*)\qquad \textup{for} \quad \ell=1,\dots, \d\,.
\label{eq:GAaniso}
\ee
\end{itemize}
\end{remark}

\section{Convergence analysis}

\label{sec:analysis}

In this section, we investigate the large-time behavior of  the kinetic process $\OX^{(k)}$ evolving according to the update rule \eqref{eq:mono}. We restrict the analysis to the case of Boltzmann Selection and mutation given by the standard Gaussian distribution ($D = I_\d$ and $\xi$ sampled according to $\mu = \mathcal{N}(0,I_\d)$ in \eqref{eq:coll}).

For any $f\in \mathcal{P}(\Rd)$, we denote its mean with
\[m[f] := \int x\, df(x)\,.\] 
We are interested in studying under which conditions the expected value $m[f_{(k)}]$ converges towards a global solution $x^\star$ to \eqref{typrob}, by looking at the large-time behavior of
\[ | m[f_{(k)}] -x^\star|  \,.\]

After stating the necessary assumptions on the objective function $\E$, we present the main theorem and a sketch of the proof. The details of the proof are then provided in the subsequent subsections.

\begin{assumption} \label{asm:E}
The objective function $\E: \Rd \to \R$ is continuous and satisfies: 
\begin{enumerate}[i)]
\item  \label{asm:unique} (solution uniqueness) there exists a unique global minimizer $x^\star$;
\item  \label{asm:growth}
(growth conditions) there exists $L_\E$, $c_u, c_l, R_l>0$ such that 
\begin{equation*}
\begin{cases}
|\E(x) - \E(y) | \leq L_\E (1+ |x| + |y|)|x - y| & \quad \forall\; x,y \in\Rd \\
\E(x) -  \E (x^\star)  \leq c_u (1 + |x|^2) & \quad \forall\; x \in\Rd \\
\E(x) - \E (x^\star) \geq c_l |x|^2 & \quad \forall\; x \,:\, |x|>R_l \,;
\end{cases}
\end{equation*}
\item  \label{asm:inverse}
(inverse continuity) there exist $c_p , p>0, R_p>0$  and lower bound $\E_\infty>0$ such that
\begin{equation*}
\begin{cases}
c_p|x - x^\star|^{p} \leq \E(x) - \E(x^\star) & \quad \forall x \,:\, |x-x^\star| \leq R_p \\
\E_\infty < \E(x) - \E(x^\star) & \quad \forall x \,:\, |x-x^\star| > R_p\,.
\end{cases}
\end{equation*}
\end{enumerate}
\end{assumption}

Before presenting the main result, let us briefly comment on the assumptions above. First of all, we note that $\E$ is only required to be continuous and it is allowed to be non-differentiable and possibly non-convex. The objective function behavior, though, needs to be controlled via upper and lower bounds. Assumption \ref{asm:E}.\ref{asm:growth}, in particular, establishes local Lipschitz continuity of $\E$ and a quadratic-like growth at infinity. This assumption will be needed to establish moments estimates for the dynamics of the individuals.
Assumption \ref{asm:E}.\ref{asm:inverse} (from which solution uniqueness follows), on the other hand, prescribes a polynomial growth around the minimizer $x^\star$. In the optimization literature, this is known as \textit{inverse continuity} \cite{fhps20-2} or \textit{p-conditioning} \cite{rosasco2017geometry} property and it will be essential to apply the so-called Laplace principle \cite{hwang1980laplace}, see below for more details.

\begin{theorem} Let $\E$ satisfy Assumption \ref{asm:E}. Let $\OX^{(0)}$ be distributed according to a given $f_{(0)} \in \mathcal{P}_2(\Rd)$ such that $x^\star \in \textup{supp}(f_{(0)})$, and $\OX^{(k)}$ be updated according to \eqref{eq:mono} with Boltzmann Selection and $D=I_\d, \overline{\xi}^{(k)} \sim \mathcal{N}(0,I_\d)$.

Fix an arbitrary accuracy $\eps>0$ and $T^\star$ be the time horizon given by
\be
T^\star := \log\left(\frac{2|m[f_{(0)}] - x^\star|}{\eps}\right)\,.
\ee
Then, there exists $\alpha>0$ sufficiently large such that
\be 
\min_{k: k\tau \leq T^\star}\left | m[f_{(k)}] - x^\star \right | \leq \eps\,.
\ee 
Furthermore, until the desired accuracy is reached, it holds
\be
|m[f_{(k)}] - x^\star | \leq e^{-k\tau} | m[f_{(0)}] - x^\star|\,.
\ee
\label{t:main}
\end{theorem}

Before delving into the details, let us outline the steps of the proof.  For notational simplicity, consider the Boltzmann-Gibbs distribution $\eta_{(k)}^\alpha$ associated with $f_{(k)}$:
\begin{equation*}
d\eta_{(k)}^\alpha(x)  := \frac{e^{-\alpha\E(x)}}{Z_{\alpha}[f_{(k)}]}\, df_{(k)}(x) \quad \textup{with} \quad Z^\alpha [f_{(k)}] := \int e^{-\alpha\E(x)}\, df_{(k)}(x)\,.
\label{eq:gibbs}
\end{equation*}
With this notation, the parents' probability is given by
$ \Par[f_{(k)}]=\eta_{(k)}^\alpha \otimes \eta_{(k)}^\alpha$,
and the law $f_{(k)}$ of the mono-particle process \eqref{eq:mono} satisfies for any measurable $\phi$
\begin{equation}
\int \phi (x) df_{(k+1)}(x) = (1 - \tau)\int \phi (x) \,df_{(k)}(x) + \tau \iiint \phi(x')\,
d\eta^\alpha_{(k)}(x)\,d\eta^\alpha_{(k)}(x_*)d\mu(\xi)\,.
\label{eq:BSweak}
\end{equation}
with $\mu = \mathcal{N}(0,I_\d)$.

With the choice $\phi(x)=x$, we note that the expected value moves towards the weighted average $m^\alpha[f_{(k)}] = m[\eta_{(k)}^\alpha]$ (already defined in \eqref{eq:wmean}) as it holds
\be
m[f_{(k+1)}] = (1- \tau) m[f_{(k)}] + \tau m^\alpha[f_{(k)}]\,.
\notag
\ee
Thanks to the quantitative version of Laplace principle (Proposition \ref{p:laplace}) derived in \cite{fornasier2021consensusbased},  it holds
\[m^\alpha[f_{(k)}] \approx x^\star \quad \textup{for} \quad \alpha \gg 1\,, \]
leading to an exponential of $m[f_{(k)}]$ towards $x^\star$. 

Let $B_r(x)$ denote the closed ball centered at $x$ of radius $r$. To apply the Laplace principle uniformly over the time window $[0,T^\star]$, though, we need to provide lower bounds on the mass around the minimizer $x^\star$. To this end, the diffusive behavior given by the mutation mechanism will be a key property of the dynamics allowing us to prove that
\[f_{(k)} (B_r(x^\star)) \geq \delta>0 \quad \textup{for all} \quad k\tau \in [0,T^\star]\,, \]
for some small radius $r>0$. In Section \ref{sec:an:moments} we collect such a lower bound, while in Section \ref{sec:an:laplace} we recall the Laplace principle and show of a proof for Theorem \ref{t:main}.

\begin{remark} $\,$
\begin{itemize}
\item 
The proof makes use of the Laplace principle and therefore it relies on the particular structure of the Boltzmann-Gibbs distribution appearing in the selection mechanism. We leave the convergence analysis for different selection mechanisms, especially those based on the particles rank, to future work.
\item
While Assumption \ref{asm:E} is similar to the one required for the analysis of CBO methods there is a main difference between Theorem \ref{t:main} and the analogous CBO convergence result in \cite{fornasier2021consensusbased}.
That is, we study the evolution of $|m[f_{(k)}] - x^\star|$, instead of the expected mean squared error
\[ \int |x - x^\star|^2 df_{(k)}(x)\,.\] 

It is the structure of the diffusive component that determines which notion of error is convenient to use for the analysis. As the mutation mechanism is non-degenerate, we cannot expect the mean squared error to become arbitrary small in genetic algorithms.
\item 
Another approach suggested in \cite{carrillo2018analytical} is also common in the analysis kinetic and mean-field models in optimization, see e.g. \cite{albi2023kinetic,benfenati2021binary,carrillo2019consensus}. The approach consists of first establishing asymptotic convergence of the particles towards a point, by studying the evolution of the system variance. Then, one investigates how far such point is from the global minimizer. 
Typically, this approach leads to better convergence results in the Fokker-Planck limit $\ve \to 0$ for the KBO kinetic model \cite{benfenati2021binary}, suggesting better performance of mean-field type dynamics even in the case of the
kinetic GA model analyzed in Section \ref{sec:CBO}.
\end{itemize}
\label{rmk:theorem}
\end{remark}

\subsection{Estimates of mass around minimizer}
\label{sec:an:moments}

Due to the mutation mechanism with Gaussian noise $\overline \xi^{(k)}$, the support of the  kinetic probability measure $f_{(k)}$ is expected to be the full search space $\Rd$ at all steps $k \geq 1$. This means, in particular, that $x^\star \in \supp(f_{(k)})$. To quantify how large the parameter $\alpha$ should be in order to apply the Laplace principle, though, we need to derive quantitative estimates of the mass around the minimizer $x^\star$. By first studying the evolution of the second moments of $f_{(k)}$ we show in the following that most of the probability mass lies on a ball $B_R(0)$ for some possibly large $R>0$. Thanks to the Gaussian mutation, we than show that $f_{(k)}(B_r(x^\star))\geq \delta$ for some possibly small $\delta,r>0$.

\begin{lemma}
\label{l:massR}
 Let $\E$ satisfy Assumption \ref{asm:E}.\ref{asm:growth} and $f_{(0)}\in \mathcal{P}_2(\Rd)$. Assume $\alpha>1$ and consider an arbitrary $\overline \delta\in (0,1)$. For a sufficiently large radius $R=R(\overline \delta, T^\star, c_u,c_l,R_l,\sigma)$ it holds
   \[ f_{(k)}(B_R(0)) \geq 1- \overline \delta \qquad \text{for all}\quad k \tau  \in [0, T^\star]\,. \]
    \end{lemma}

\begin{proof} We restart by recalling the following estimate from \cite[Lemma 3.3]{carrillo2018analytical}. For any $f \in \mathcal{P}_2(\Rd)$ with correspondent Boltzmann-Gibbs distribution $\eta^\alpha$, it holds
\[\int |x|^2d\eta^{\alpha}(x) \leq b_1 + b_2 \int |x|^2 df(x) \] 
with
\[ b_1 = R_l  +b_2, \quad b_2 = 2 \frac{c_u}{c_l}\left(1 + \frac{1}{\alpha c_l}\frac1{R_l^2}   \right)  \,.\]
We note that the above estimate can be made independent on $\alpha$ since we assumed $\alpha>1$.
Consider now $f_{(k)}$, $k \geq 1$. From the weak formulation \eqref{eq:BSweak}
\begin{align*}
\int |x|^2 df_{(k)}(x)  &= (1-\tau)\int|x|^2 df_{(k-1)}(x) + \tau \iiint |x'|^2 d\eta^\alpha_{(k-1)}(x)d\eta^\alpha_{(k-1)}(x_*) d\mu(\xi) \\
&\leq (1-\tau)\int|x|^2 df_{(k-1)}(x) + 4\tau \int |x|^2 d\eta^\alpha_{(k-1)}(x) + 4\sigma^2 \tau \\ 
&\leq (1-\tau)\int|x|^2 df_{(k-1)}(x) + 4\tau \left(b_1 + b_2 \int |x|^2 df_{(k-1)}(x)\right) + 4\sigma^2 \tau  \\
& =: (1 + q_1\tau)\int|x|^2 df_{(k-1)}(x) + q_2 \tau
\end{align*}
with $q_1 = 4b_2 - 1$ and $q_2= 4 b_1 + 4\sigma^2$. By iteratively applying the above estimate, we can derive a discrete G{\"o}nwall-like inequality
\begin{align*}
\int|x|^2 df_{(k)}(x) & \leq (1 + q_1 \tau)^k \int|x|^2 df_{(0)}(x) + q_2 \tau \sum_{h=0}^{k-1} (1 + q_1 \tau)^h \\
 & = (1 + q_1 \tau)^k \int|x|^2 df_{(0)}(x) + \tau q_2 \frac{1 - (1+q_1 \tau)^k}{1 - (1+q_1\tau)}\\
 & = (1 + q_1 \tau)^k \int|x|^2 df_{(0)}(x) + \frac{q_2}{q_1}( (1+q_1 \tau)^k - 1  )\\
 & \leq e^{k\tau q_1} \int|x|^2 df_{(0)}(x)+ \frac{q_2}{q_1}\left(e^{k\tau q_1} - 1 \right )\,.
\end{align*}
For any $R>0$, by Markov's inequality it holds
\[ 
\mathbb{P}(|\overline{X}_{(k)}| \geq R) \leq \frac1{R^2} \mathbb{E}[|\overline{X}_{(k)}|^2]\,,
\] 
leading to
\begin{align*}
\mathbb{P}(|\overline{X}_{(k)}| < R) 
  \geq 1  - \frac1{R^2} \mathbb{E}[|\overline{X}_{(k)}|^2]  \geq  1 - \frac1{R^2} \left(e^{T^\star q_1} \int|x|^2 df_{(0)}(x)+ \frac{q_2}{q_1}\left(e^{T^\star q_1} - 1 \right ) \right)
\end{align*}
for all $k\tau\in [0,T^\star]$. We conclude by taking $R>0$ sufficiently large with respect to $\overline \delta$.
\end{proof}

Next, we provide an analogous lower bound for $\eta^\alpha_{(k)}$, that is, the Gibbs-Boltzmann probability measure associated with $f_{(k)}$ and the objective function $\E$.

\begin{lemma} 
\label{l:gibbsmassR}
Let $\E$ satisfy Assumption \ref{asm:E}.\ref{asm:growth} and let $f_{(k)}(B_R(0))\geq 1- \overline\delta$ for all $k\tau\in [0,T^\star]$ for some $R>R_l, \overline \delta >0$. Then, there exists $R',\alpha$ sufficiently large which depend on $\{\overline \delta, T^\star, c_u,c_l,R_l,\sigma\}$ such that 
\[  \eta^\alpha_{(k)} \left( B_{R'}(0) \right) \geq 1- \overline{\delta} \qquad \text{for all}\quad  k \tau \in[0, T^\star]\,. \] 
\end{lemma}

\begin{proof} We note that the following computations are similar to the ones performed in \cite[Lemma 3.3]{carrillo2018analytical} and \cite[Proposition 21]{fornasier2021consensusbased}. By Markov's inequality we have
\begin{align*}
Z^\alpha[f_{(k)}]& \geq \exp \left(-\alpha \sup_{x \in B_R(0)}\E(x)\right) f_{(k)}\left(\left \{x\,:\, \E(x) \leq \sup_{x \in B_R(0)}\E(x)\right \} \right)  \\
& \geq  \exp\left(-\alpha \sup_{x \in B_R(0)}\E(x) \right )  f_{(k)} \left(B_R(0) \right)\,.
\end{align*}

For any $R'>R$, it holds for all $k\tau \in [0, T^\star]$
\begin{align*}
\eta^\alpha_{(k)} \left( B_{R'}^c (0) \right)  & = \int_{B_{R'}^c} \frac{e^{-\alpha \E(x)}}{Z^\alpha[f_{(k)}]}df_{(k)}(x)\\
&\leq  \frac{ \int_{B_{R'}^c}  e^{-\alpha \E(x)}df_{(k)}(x)}{ \exp\left(-\alpha \sup_{x \in B_R(0)}\E(x) \right )  f_{(k)} \left(B_R(0) \right)}\\
&\leq  \frac { \exp(-\alpha \left (\inf_{x \in B^c_{R'}(0)}\E(x)\right) f_{(k)}\left(B_{R'}^c(0)\right)}
{ \exp\left(-\alpha \sup_{x \in B_R(0)}\E(x) \right )  f_{(k)} \left(B_R(0) \right)} \\
& \leq \frac{1}{1-\overline{\delta}} \exp\left (\alpha \left(\sup_{x \in B_R(0)}\E(x) -\inf_{x \in B^c_{R'}(0)}\E(x)  \right)   \right)
\end{align*}
where we used $f_{(k)}(B_{R'}^c) \leq 1$ and the assumption $f_{(k)} \left(B_R(0) \right) \geq 1-\overline \delta$ .

Thanks to the growth conditions in Assumption \ref{asm:E}.\ref{asm:growth} and for $R'> R> R_l$, it holds
\begin{align*} 
\sup_{x \in B_R(0)}\E(x) -\inf_{x \in B^c_{R'}(0)}\E(x)  & = \sup_{x \in B_R(0)}\E(x) -\E(x^\star)+ \E(x^\star)-\inf_{x \in B^c_{R'}(0)}\E(x)
\\
& \leq c_u ( 1 + R^2) - c_l (R')^2   \leq - 1
\end{align*}
for $R'$ sufficiently large with respect to $R, c_u$, and $c_l$. This leads to 
\[\eta^\alpha_{(k)} \left( B_{R'}^c (0) \right) \leq \frac{e^{-\alpha}}{1- \overline\delta}\,. \] 
By taking $\alpha$ sufficiently large with respect to $\overline \delta$ we obtain the desired lower bound for $\eta^\alpha_{(k)} ( B_{R'} (0))$ for all $k\tau \in [0,T^\star]$.
\end{proof}

Finally, we are ready to give a quantitative estimate for the mass in a neighborhood of the minimizer $x^\star$.

\begin{proposition}
\label{p:mass}
Under the assumptions of Theorem \ref{t:main}, fix a radius $r>0$. 
Provided $\alpha=\alpha(r)$ is sufficiently large, it holds
\begin{equation} \notag
f_{(k)}(B_r(x^\star)) \geq  \min\left\{f_{(0)}(B_r(x^\star))\,,\, \delta_r\right\}\,,  \quad \text{for all}\quad k\tau \in[0,T^\star]
\end{equation}
and for some positive $\delta_r =\delta_r(r,\d, T^\star,c_u,c_l,R_l,\sigma) $.
\end{proposition}

\begin{proof}
For a given set $A \subset \Rd$, let $\mathbf{1}_A$ denote its indicator function. From the weak formulation \eqref{eq:BSweak} with $\phi = \mathbf{1}_{B_r(x^\star)}(x)$, we get 
\begin{equation}
\label{eq:weakindicator}
f_{(k+1)}(B_r(x^\star)) =(1- \tau)f_{(k)}(B_r(x^\star)) + \tau \iiint \mathbf{1}_{B_r(x^\star)}(x')d\eta^\alpha_{(k)}(x) d\eta^\alpha_{(k)}(x_*) d\mu(\xi) \,.
\end{equation} 
For the couple of parents $(x,x_*)\in\Rd\times \Rd$, we introduce the set
\[ \Xi_{x,x_*} = \{\xi \in \Rd\;:\;x' \in B_r(x^\star)   \}\,.
\] 

Next, we apply Lemma \ref{l:massR} and Lemma \ref{l:gibbsmassR} with, for simplicity, $\overline\delta= 1/2$ to obtain $R, R'$ and $\alpha$ sufficiently large such that
\[f_{(k)}(B_R(0))\,, \;\; \eta_{(k)}^\alpha(B_{R'}(0)) \geq \frac 12 \qquad \textup{for all}\quad  k\tau\in [0,T^\star]\,. \] 
By eventually taking an even larger radius, we assume $R' > R_l$ and $R' >|x^\star| + r$.

Let restrict to parents' couples with $x,x_* \in B_{R'}(0)$. By definition of $x'$ \eqref{eq:coll} condition $x' \in B_r(x^\star)$ can written as
\[  \sigma \xi \in B_r( x^\star  - (1 - \gamma) \odot x - \gamma \odot x_* ) =: B_r( \tilde x ) \]
with $|\tilde x| \leq 2R'$.
Since $ \mu = \mathcal{N}(0, I_\d)$, we have that $\mu$ is absolutely continuous with respect to the Lebesgue measure $\mathcal{L}^\d$, with density $e^{-(1/2)|\xi|^2}/\sqrt{2\pi \d}$. Thanks to the choice $R'> |x^\star| + r$, we have
\begin{align*}
\mu(\Xi_{x,x_*}) &= \frac{1}{\sqrt{2\pi \d}}  \int \mathbf{1}_{ B_r(\tilde x) }(\sigma \xi ) e^{- |\xi|^2/2}\, d\xi \\
& \geq  \frac{e^{- (R')^2/2}}{\sqrt{2\pi \d }}  \int \mathbf{1}_{ B_{r/\sigma}(\tilde x/\sigma)}(\xi) \,d\xi\\
& = \frac{e^{-(R')^2/2}}{\sqrt{2\pi \d }} \mathcal{L}^\d\left( B_{r/\sigma}(0)\right)\,.
\end{align*}

Next, we apply Fubini's Theorem to the last term in \eqref{eq:weakindicator}, to get
\begin{align*}
\iiint \mathbf{1}_{\{x' \in B_r(x^\star) \}}d\eta^\alpha_{(k)}(x) d\eta^\alpha_{(k)}(x_*)d\mu(\xi) &
 = \iint \mu(\Xi_{x,x_*}) d\eta^\alpha_{(k)}(x) d\eta^\alpha_{(k)}(x_*)  \\
 &
 \geq \int_{B_{R'}(0)} \int_{B_{R'}(0)} \mu(\Xi_{x,x_*}) d\eta^\alpha_{(k)}(x) d\eta^\alpha_{(k)}(x_*)  \\
 & \geq \frac{e^{-(R')^2/2}}{\sqrt{2\pi \d }} \mathcal{L}^\d\left( B_{r/\sigma}(0)\right)\eta_{(k)}^\alpha(B_{R'}(0))^2 \\
 & \geq \frac{e^{-(R')^2/2}}{\sqrt{2\pi \d }} \mathcal{L}^\d\left( B_{r/\sigma}(0)\right) \frac14 =:\delta_r
\end{align*}
for some $\delta_r \in (0,1)$, where we used $\eta^\alpha_{(k)}(B_{R'}(0)) \geq 1/2$ in the last step. 
By inserting the above estimate in \eqref{eq:weakindicator} we have 
\[ f_{(k+1)}(B_r(x^\star)) \geq (1-\tau)f_{(k)}(B_r(x^\star)) + \tau \delta_r \geq \min\left\{f_{(k)}(B_r(x^\star))\,,\, \delta_r\right\}\,,
\]
from which the claim follows.

We note that $\delta_r$ explicitly depends on the quantities $r,\d,\sigma,R'$ and so, consequently on the set $\{r,\d, T^\star,c_u,c_l,R_l,\sigma \}$. Most importantly, we remark that $\delta_r$ does not depend on $\alpha$.
\end{proof}

\subsection{Laplace principle and proof of Theorem \ref{t:main}}
\label{sec:an:laplace}

The reason why we focused on GA with Boltzmann Selection for the convergence analysis lies in the theoretical properties of the Boltzmann-Gibbs distribution. In particular, we aim to exploit the following asymptotic result, known as the \textit{Laplace principle} \cite{Dembo2010}. For any $f\in \mathcal{P}(\Rd)$ absolutely continuous with respect to the Lebesgue measure and with global minimizer $x^\star \in \supp(f)$, it holds
\[
\lim_{\alpha \to \infty} \left( -\frac1\alpha \log \left(  \int e^{-\alpha \E(x)}df(x) \right) \right) = \E(x^\star)\,.
\]
Intuitively, this means that the Boltzmann-Gibbs distribution $\eta^\alpha = e^{-\alpha\E}f/Z^\alpha[f]$ associated with $f$ approximates $\delta_{x^\star}$ for $\alpha \gg 1$. In \cite{fornasier2021consensusbased}, the authors derive a quantitative version of the Laplace principle which allows to control the distance between the weighted mean $m^\alpha[f]$ from the minimizer $x^\star$ through the parameter $\alpha$. We recall the result for completeness.

\begin{proposition}{\cite[Proposition 21]{fornasier2021consensusbased}}
\label{p:laplace}
Let $\inf \E = 0$, $f \in \mathcal{P}(\Rd)$ and fix $\alpha >0$. For any $r>0$ we define $\E_r:= \sup_{x \in B_r(x^\star)} \E(x)$. Then, under Assumption \ref{asm:E}.\ref{asm:inverse}, for any $r \in (0, R_p]$ and $q>0$ such that $q + \E_r< \E_\infty$, we have
\be
|m^\alpha [f] - x^\star| \leq c_p (q + \E_r)^{1/{p}} + \frac{\exp(-\alpha q)}{f(B_r(x^\star))} \int |x - x^\star| df(x)\,.
\label{eq:laplace}
\ee
\end{proposition}

We note that the assumption $\inf \E=0$ is actually superfluous as the Boltzmann-Gibbs weights $e^{-\alpha\E}/Z^\alpha[f]$ are invariant with respect to translations of $\E$. 
From \eqref{eq:laplace}, it becomes clear that quantifying the mass around the minimizer is essential to make the right-hand-side arbitrary small as $\alpha \to \infty$. This is the motivation behind the derivation of lower bounds for $f_{(k)}(B_r(x^\star))$ for all $k\in [0, T^\star]$ in the previous subsection.

Finally, by combining Propositions \ref{p:mass} and \ref{p:laplace} we provide a proof to the main Theorem.

\begin{proof}[Proof of Theorem \ref{t:main}] As noted before, the expected value $m[f_{(k)}]$ evolves according to
\[m[f_{(k+1)}] = (1- \tau) m[f_{(k)}] + \tau m^\alpha[f_{(k)}]\,,\]
from which follows
\begin{align}
\label{eq:errest1}
|m[f_{(k)}] - x^\star| \leq (1-\tau)|m[f_{(k-1)}] - x^\star| + \tau |m^\alpha[f_{(k-1)}] - x^\star|\,.
\end{align}

Recall the time horizon $T^\star$ is given by $T^*:=\log(2|m[f_{(0)}] - x^\star|/\eps)$.
We apply the quantitative Laplace principle (Proposition \ref{p:laplace}) for all $k\tau\in [0,T^\star]$ with  $q_\eps, r_\eps$ be given by
\[ q_\eps:= \frac12 \min\left\{\left(\frac{\eps}{4 c_p} \right)^p, \E_\infty \right \}\,, \quad r_\eps := \max_{s\in[0,R_p]} \{\E_s =  \sup_{x\in B_s(x^\star)} \E(x)  \leq q_\eps \}\,,
\] 
leading to 
\begin{align*}
|m^\alpha[f_{(k)}] - x^\star| & \leq c_p (q_\eps + \E_{r_\eps})^{1/{p}} + \frac{\exp(-\alpha q_\eps)}{f_{(k)}(B_{r_\eps}(x^\star))} \int |x - x^\star| df_{(k)}(x) \\
& \leq \frac \eps 4 + \frac{\exp(-\alpha q_\eps)}{f_{(k)}(B_{r_\eps}(x^\star))} \int |x - x^\star| df_{(k)}(x)\,,
\end{align*}
thanks to the choice of $q_\eps$.

To bound the second term, we apply Proposition \ref{p:mass} and use the energy estimates derived in the proof of Lemma \ref{l:massR}:
\begin{align*}
\int |x - x^\star| df_{(k)}(x) &\leq |x^\star| +  \left(\int |x|^2 \right)^{1/2} \\
&\leq |x^\star| +  \left(e^{T^\star q_1} \int|x|^2 df_{(0)}(x)+ \frac{q_2}{q_1}\left(e^{T^\star q_1} - 1 \right ) \right)^{1/2}=: |x^\star| + K \,,
\end{align*}
and, for $\alpha$ sufficiently large,
\[f_{(k)}(B_{r_\eps}(x^\star)) \geq \min \left \{f_{(k)}(B_{r_\eps}(x^\star)), \delta_r \right \}\,.\] 

By eventually taking an even larger $\alpha$ sufficiently, we obtain
\begin{align*}
|m^\alpha[f_{(k)}] - x^\star| & \leq \frac \eps 4 +\frac{\exp(-\alpha q_\eps)}{f_{(k)}(B_{r_\eps}(x^\star))} \int |x - x^\star| df_{(k)}(x) \\
& \leq \frac \eps 4 + \frac{\exp(-\alpha q_\eps)}{\min\{f_{(0)}(B_{r_\eps}(x^\star)), \delta_{r_\eps}\}} \left (|x^\star| +  K \right) \\
& \leq  \frac \eps 4 + \frac \eps 4  \leq \frac12 |m[f_{(k)}] - x^\star|\,.
\end{align*}
as long as $|m[f_{(k)}] - x^\star|> \eps$.

By plugging the above in \eqref{eq:errest1}, we obtain
\begin{align*}
|m[f_{(k)}] - x^\star| &  \leq (1 - \tau/2)^{k}|m[f_{(0)}] - x^\star| \leq  e^{-k\tau/2}|m[f_{(0)}] - x^\star| \,,
\end{align*}
until the desired accuracy is reached. By definition of $T^\star$, we get
\[\min_{k\,:\,k\tau<T^\star}|m[f_{(k)}] -x^\star| \leq \eps \,.\]  
\end{proof}

\section{Numerical experiments}
\label{sec:num}

In this section, we numerically investigate several aspects related to the proposed kinetic approximation of the GA algorithm.
First of all, we examine how the particles' distribution $f^N_{(k)}$ varies when different population sizes $N$ are considered. We are interested, in particular, in validating the propagation of chaos assumption by numerically checking if $f^N_{(k)}$ converges to a deterministic probability measure as $N$ increases.
Next, we numerically verify to which configurations the GA particle system converges asymptotically for different objectives and selection mechanisms. 
Finally, we compare GA with CBO under the settings of Section \ref{sec:CBO} both with non-degenerate and anisotropic exploration. We do so by considering different scaling parameter $\ve$ and different benchmark problems.

\subsection{Validation of propagation of chaos assumption and steady states}

\begin{figure}\centering
\begin{subfigure}{1\linewidth}
\includegraphics[width = 1\linewidth]{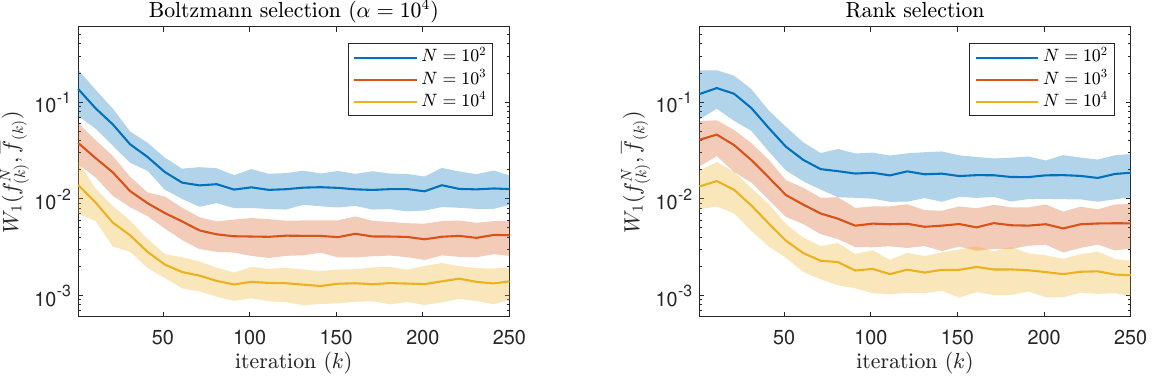}
\caption{Constant mutation strength $\sigma = 0.1$}
\label{fig:propagation_chaos}
\end{subfigure} 
\medskip

\begin{subfigure}{1\linewidth}
\includegraphics[width = 1\linewidth]{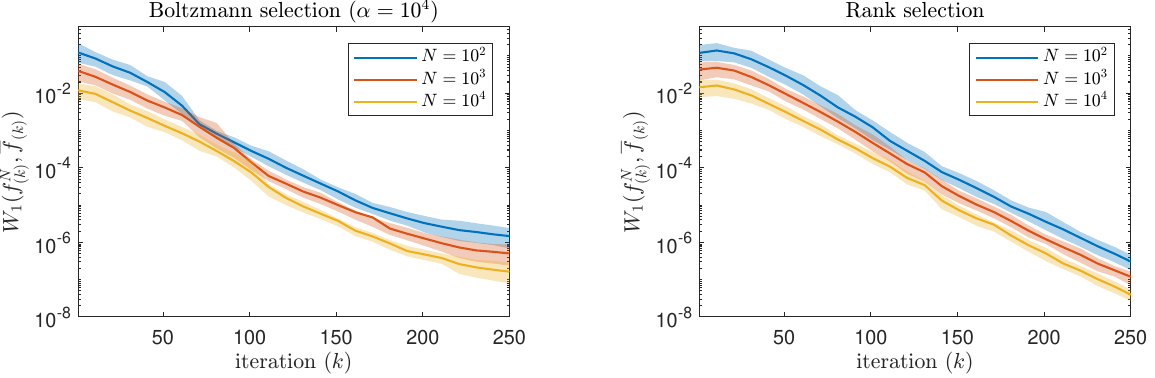}
\caption{Decreasing mutation strength $\sigma_{(0)} = 0.1$, $\sigma_{(k+1)} = 0.95\sigma_{(k)}$ (cooling strategy)}
\label{fig:propagation_chaos_2}
\end{subfigure}

\caption{Evolution of 1-Wasserstein distance between $f^N_{(k)}$ and a reference solution $\overline{f}_{(k)}$ obtained with $10^5$ particles and that approximates the kinetic evolution. Different population sizes $N$ and selection mechanisms are considered. Continuous lines show the mean over $100$ realizations, while shaded area indicate the [0.1,0.9] quantile interval. The objective considered is the Ackley function. Parameters are set to $\tau = 0.1, \gamma = 0.2 (1, \dots, 1)^\top$, different mutation strategies are considered,  and the initial particles' locations are independently sampled from $\textup{Unif}[-2,2]$.}
\label{fig:propagation_chaos}
\end{figure}

We recall that every run of the GA algorithm corresponds to a realization of the Markov process \eqref{eq:Nevolution}. The empirical distribution $f^N_{(k)}$ of the particle system $\{ X_i^{(k)}\}_{i=1}^N$ at step $k$ is therefore a random variable taking values in the space $\mathcal{P}(\Rd)$. Assuming propagation of chaos, as we have done to derive the kinetic mono particle process \eqref{eq:mono}, corresponds to assuming that $f^N_{(k)}$ converges to a deterministic probability measure $f_{(k)} \in \mathcal{P}(\Rd)$ as $N \to \infty$.

To investigate whether this is a plausible assumption, we consider a reference solution $\overline{f}_{(k)}$ obtained with a relatively large population size of $10^5$ particles. Due to the large number of particles, we consider $\overline{f}_{(k)}$ to approximate the kinetic measure $f_{(k)}$. Then, we collect different realizations of $f^N_{(k)}$ with an increasing number of particles $N = 10^2, 10^3, 10^4$ by running the GA algorithm 100 times for each population size. To quantify the distance between the realizations and the reference solution, we use the 1-Wasserstein distance $W_1(\cdot, \cdot)$ which metrizes weak convergence in $\mathcal{P}_1(\Rd)$. For a definition of Wasserstein distances and their properties we refer to \cite{santambrogio2015optimal}.
As optimization problems, we consider the minimization of the Ackley objective function \cite{JamilY13} over $\RR$, that is, we set $\d = 1$.

Figure \ref{fig:propagation_chaos} shows the results for GA algorithm with Boltzmann Selection (BS) and Rank Selection (RS), respectively. For a fixed step $k$, we note that the mean distance among the realizations decreases as $N$ increases, suggesting an asymptotic decay of $\mathbb{E}W_1(f_{(k)}^N, f_{(k)})$ as $N \to \infty$. Moreover, the [0.1,0.9] quantile interval size also decreases as $N$ increases (note that logarithmic scale is used in the plots for $y$-axis), suggesting that $f^N_{(k)}$ indeed converges to a deterministic probability measure. These numerical experiments, therefore, support the propagation of chaos assumption.

If we look at the evolution of the 1-Wasserstein distance over the time iteration $k$, it is interesting to note that the distance decreases over time, until a certain level is reached which depends on the population size $N$. Remarkably, this suggests that the error introduced by the kinetic approximation $f^N_{(0)}\approx f_{(0)}$ does not increment during the computation, but rather it decreases making the kinetic model more and more accurate as the particles move towards the problem's solution. This behavior may be explained by looking at the configurations that large particles systems reach asymptotically.

Figure \ref{fig:steady_states} shows final configurations reached by a particle system with size $N = 10^6$ after $k = 10^3$ iterations. As expected due to the mutation mechanism which adds constant Gaussian noise to the dynamics, the particles seem to normally distribute in a neighborhood the global minimizer. 
To explain why the approximation error decreases during the evolution,  we conjecture that, even with a relatively small population size $N$, the GA particle system is able to capture the behavior of the kinetic evolution and to convergence to an approximation of the same steady states.  

If we consider a GA dynamics with decreasing mutation strength ($\sigma_{(k)} \to 0$ as $k \to \infty$) the expected asymptotic configurations as $k\to \infty$ are given by Dirac measures. As a Dirac measure can be described even by a particle system with only $N = 1$ particle, we expect the kinetic approximation to become exact in the limit $k \to \infty$, provided the GA particle system converges to the global minimum. This is indeed confirmed by numerical experiments. Figure \ref{fig:propagation_chaos_2} shows the evolution of $W_1(f_{(k)}^N, \overline{f}_{(k)})$ when the diffusive behavior given by the mutation mechanism is decreasing. We note how the particles systems $f^N_{(k)}$ converges to the reference solution $\overline{f}_{(k)}$ not only as $N$ increases, but also as $k \to \infty$.

Finally, we remark that, the closeness of the mean of the asymptotic configuration to the minimizer $x^\star$ depends on the effectiveness of the selection mechanism.   For instance, larger values of the parameter $\alpha$ for BS lead to asymptotic configuration more centered around the minimizer, see Figure \ref{fig:steady_states:boltzmann10} and  \ref{fig:steady_states:boltzmann10000} (right plots). This is consistent with the theoretical analysis for GA with BS performed in Section \ref{sec:analysis}.

\begin{figure}
\centering
\begin{subfigure}{0.45\linewidth}
\includegraphics[trim = 0 3.5cm 5cm 0,clip]{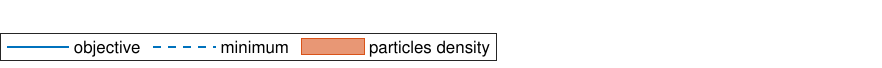}
\end{subfigure}
\begin{subfigure}{1\linewidth}
\includegraphics[width = 1\linewidth]{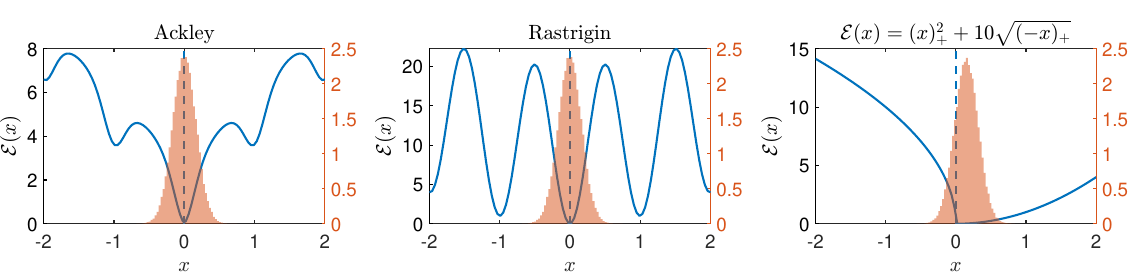}
\caption{Boltzmann Selection ($\alpha = 10$)}
\label{fig:steady_states:boltzmann10}
\end{subfigure}
\begin{subfigure}{1\linewidth}
\includegraphics[width = 1\linewidth]{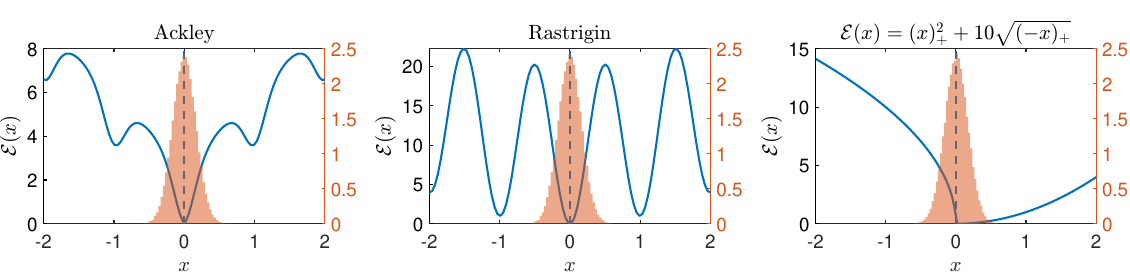}
\caption{Boltzmann Selection ($\alpha = 10^4$)}
\label{fig:steady_states:boltzmann10000}
\end{subfigure}
\begin{subfigure}{1\linewidth}
\includegraphics[width = 1\linewidth]{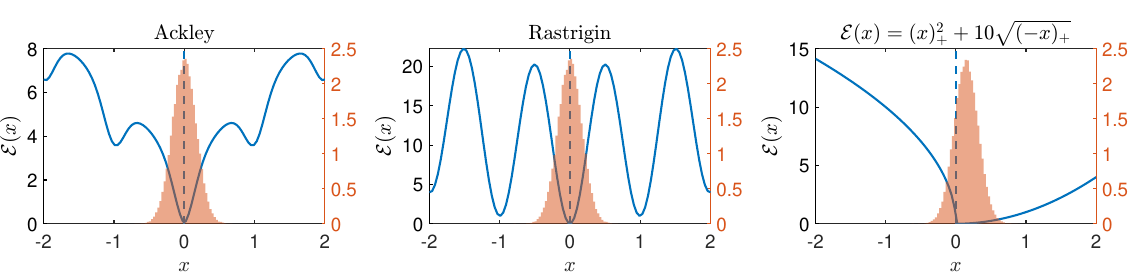}
\caption{Rank Selection}
\label{fig:steady_states:rank}
\end{subfigure}
\caption{Final particle distributions for $N = 10^6$ particles after $k_{\max} = 10^3$ iterations, for different objectives and selection mechanisms. Parameters are set to $\tau = 0.1, \sigma = 0.1, \gamma = 0.2 (1, \dots, 1)^\top$. The initial particles' locations are independently sampled from $\textup{Unif}[-2,2]$. The black and red check marks
on the $y$-axis refer to the objective function and probability distribution, respectively.
}
\label{fig:steady_states}
\end{figure}

\subsection{Comparison between GA and CBO in the scaling limit}

In this section, we compare the performance of GA in the Fokker-Planck scaling described in Section \ref{sec:CBO} for benchmark problems in dimension $\d = 10$. In particular, we consider selection mechanism and the parameters' choice which leads to the CBO dynamics as the rescaling parameter $\ve$ tends to $0$. We recall this is given by choosing $\gamma = (\lambda, \dots, \lambda)$ ($\lambda>0$ being the drift parameter in the CBO dynamics) and the parents' probability distribution 
$\Par_{CBO}[f^N_{(k)}]$, see \eqref{eq:PCBO}. If we sample the parents couple $(\tilde X_{i}^{(k)}, \tilde X^{(k)}_{*,i})$ according to $\Par_{CBO}[f^N_{(k)}]$, we have that $\law(\tilde X_{i}^{(k)}) = f^N_{(k)}$, while the second parents is sampled according to the  Boltzmann-Gibbs distribution associated with $f^N_{(k)}$, which we denote with $\eta_{(k)}^{\alpha,N}$. To obtain the CBO algorithm in the Fokker Planck asymptotic regime, instead of sampling the first parent $\tilde X_{(k)}^i$ from $f^N_{(k)}$, we directly consider the first parent to be  $X^{(k)}_i$. We remark that, form the kinetic point of view, these two choices are equivalent, as the particles are assumed to be identically distributed.

In these settings, the GA dynamics \eqref{eq:Nevolution} reads
\begin{equation*}
X^{(k+1)}_i  = (1 - T^{(k)}_{i})  X^{(k)}_i + T^{(k)}_{i}\, \left((1 - \lambda) X_i^{(k)} +  \lambda \tilde X^{(k)}_{*,i} + \sigma D \xi_i^{(k)}\right) 
\end{equation*}
with $\law(\tilde X^{(k)}_{*,i}) = \eta^{\alpha,N}_{(k)}$ and $\law(T_i^{(k)}) = \textup{Bern}(\tau)$. For convenience, we rewrite the above as
\begin{equation}
X^{(k+1)}_i  = (1 - T^{(k)}_{i})  X^{(k)}_i + T^{(k)}_{i}\, \left(X_i^{(k)} +  \lambda ( \tilde X^{(k)}_{*,i}  - X_i^{(k)})  + \sigma D \xi_i^{(k)}\right) \,.
\label{eq:GA2CBO}
\end{equation}

To implement the scaling $\eqref{eq:scaling}$ for $\ve>0$ we introduce the random variables $T^{(k)}_{\ve,i}\sim \textup{Bern}(\tau/\ve)$, where $\tau/\ve$ is the new interaction frequency. In the time-discrete settings we are considering, we note that it must hold $\tau \leq \ve$. The scaled dynamics reads
\begin{equation}
X^{(k+1)}_i = (1 - T^{(k)}_{\ve,i})  X^{(k)}_i + T^{(k)}_{\ve,i}\, \left(X_i^{(k)} +  \ve\lambda ( \tilde X^{(k)}_{*,i}  - X_i^{(k)})  + \sqrt{\ve}\sigma D \xi_i^{(k)}\right) \,.
\label{eq:GA2CBOeps}
\end{equation}

For $\ve = 1$, \eqref{eq:GA2CBOeps} is clearly equivalent to \eqref{eq:GA2CBO}, while for the smallest choice of $\ve$ allowed, $\ve = \tau$, we obtain the CBO-like dynamics
\begin{equation}
X^{(k+1)}_i =  X_i^{(k)} +  \tau \lambda ( \tilde X^{(k)}_{*,i}  - X_i^{(k)})  + \sqrt{\tau}\sigma D \xi_i^{(k)} \,.
\label{eq:GA2CBOtau}
\end{equation}
Indeed, the only thing that differentiates \eqref{eq:GA2CBOtau} from the CBO iteration is that in latter $\tilde X^{(k)}_{*,i} \sim \eta^{\alpha,N}_{(k)}$ is substituted by its expected value $\mathbb E \tilde X^{(k)}_{*,i}  = m[\eta_{(k)}^{\alpha,N}] =  m^\alpha[f^N_{(k)}]$, see \eqref{eq:wmean}. Again, we note that this difference disappears when deriving the mean-field approximation \eqref{eq:mfCBO}.

In the experiment, we aim to test the above GA dynamics for $\ve = 1, 0.3, \tau$, where $\tau = 0.1$ as before. We consider two different exploration processes: the classical GA one where the diffusion is non-degenerate (also for CBO), and the anisotropic one 
 which is potentially degenerate, see \eqref{eq:GAaniso} and \eqref{eq:CBOaniso}. In particular, we consider $D = I_\d$ with decreasing $\sigma_{(k+1)} = 0.95\sigma_{(k)}$, and $D = \textup{diag}(\tilde X_{*,i}^{(k)} - X_i^{(k)})$ with constant $\sigma$.  The remaining parameters are set to $\lambda = 1, \sigma_{(0)} = \sigma = 1$, $k_{\max} = 300$. The CBO dynamics is also included for comparison, with $\sigma = 3$. The diffusion parameters $\sigma$ have been optimized for each algorithm to provide the best performance across the benchmark tests. As optimization problems, we consider the minimization of three well known non-convex objectives in dimension $\d = 10$, that is, the functions Stydnlinski-Tank, Ackley, and Rastrigin \cite{JamilY13}.

Figures \ref{fig:scaling:error} shows the accuracy reached the end of the computation in terms of the $\ell_2$-distance between the best particle $X^{(k)}_{\textup{best}}$ and the solution $x^\star$, at the end time $k = k_{\max}$. In Figure \ref{fig:scaling:accuracy}, instead, we report the final accuracy $\E(X^{(k)}_{\textup{best}}) -\E(x^\star)$. Different population sizes are considered: $N = 10^2, 10^3, 10^4$.

We note that the algorithms' performance may vary sensibly even when the same parameters are used. This is particular evident in the minimization of the Styblinski-Tank objective where the particles  often converge towards sub-optimal solution far from the global one $x^\star$. Anisotropic exploration leads to more stable result with respect to non-degenerate exploration coupled with the cooling strategy. Typically, using larger population sizes $N$ lead to better results. Also, the GA algorithm seems to perform better for small scaling parameters $\ve$, in agreement with what observed for other kinetic models in optimization, see Remark \ref{rmk:theorem}.  As expected, the choice $\ve = \tau = 0.1$ for GA leads to similar results to the CBO algorithm.

\begin{figure}
\centering
\begin{subfigure}{0.45\linewidth}
\includegraphics[trim = 10.5cm 5.4cm 0cm 0,clip]{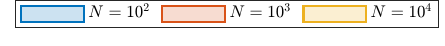}
\end{subfigure}

\bigskip

\begin{subfigure}{1\linewidth}
\includegraphics[trim = 0cm 0cm 0cm 0.45cm,clip,width = 1\linewidth]{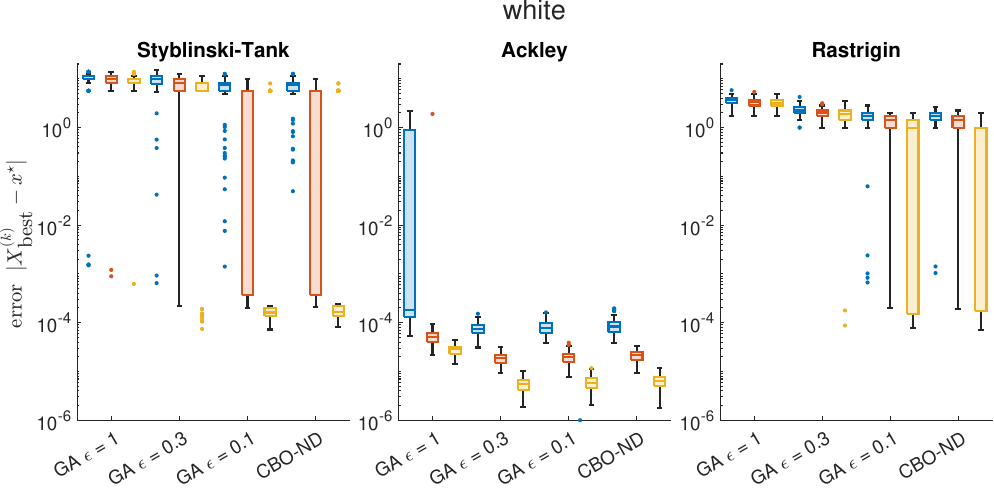}
\caption{Non-degenerate diffusion ($D = (1,\dots,1)^\top$) with cooling strategy $\sigma_{(k+1)} =0.95 \sigma_{(k)}$ }
\label{fig:scaling:error:a}
\end{subfigure}

\bigskip

\bigskip

\begin{subfigure}{1\linewidth}
\includegraphics[trim = 0cm 0cm 0cm 0.45cm,clip,width = 1\linewidth]{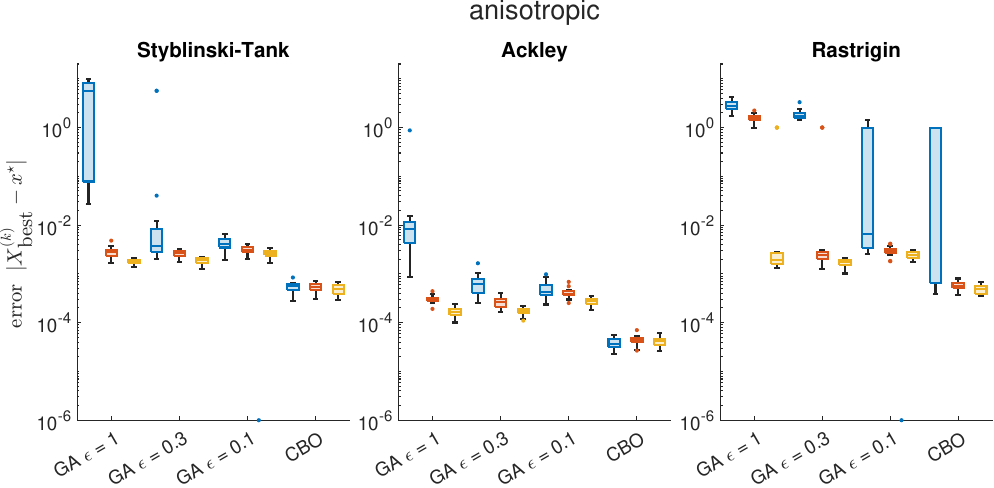}
\caption{Anisotropic diffusion ($D = x_* - x$) without cooling strategy $\sigma_{(k)} = \sigma_{(0)}$}
\label{fig:scaling:error:b}
\end{subfigure}

\caption{Error $|X^{(k)}_\textup{best} - x^\star|$ at the last iteration $k = k_{\max} = 300$ for algorithms \eqref{eq:GA2CBO} ($\ve = 1$), \eqref{eq:GA2CBOeps} ($\ve  =0.3$), \eqref{eq:GA2CBOtau} ($\ve  =\tau  = 0.1$) and CBO. Problems' dimension is $\d =10$. Parameters are set to $\tau = 0.1, \alpha = 10^4, \lambda = 1, \sigma_{(0)} = 1$. the initial particles' locations are independently sampled from $\textup{Unif}[-2,2]^\d$. The box charts display median, lower and upper quartiles, and outliers (computed with interquartile range), for the 100 experiments performed. }
\label{fig:scaling:error}
\end{figure}

\begin{figure}
\centering
\begin{subfigure}{0.45\linewidth}
\includegraphics[trim = 10.5cm 5.4cm 0cm 0,clip]{fig_scaling_comparison_legend.pdf}
\end{subfigure}

\bigskip

\begin{subfigure}{1\linewidth}
\includegraphics[trim = 0cm 0cm 0cm 0.45cm,clip,width = 1\linewidth]{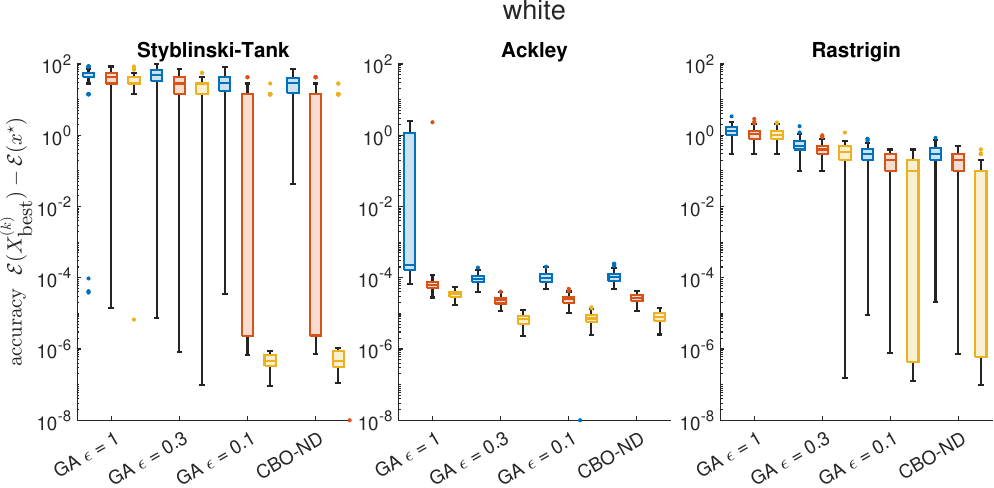}
\caption{Non-degenerate diffusion ($D = (1,\dots,1)^\top$) with cooling strategy $\sigma_{(k+1)} =0.95 \sigma_{(k)}$ }
\label{fig:scaling:accuracy:a}
\end{subfigure}

\bigskip

\bigskip

\begin{subfigure}{1\linewidth}
\includegraphics[trim = 0cm 0cm 0cm 0.45cm,clip,width = 1\linewidth]{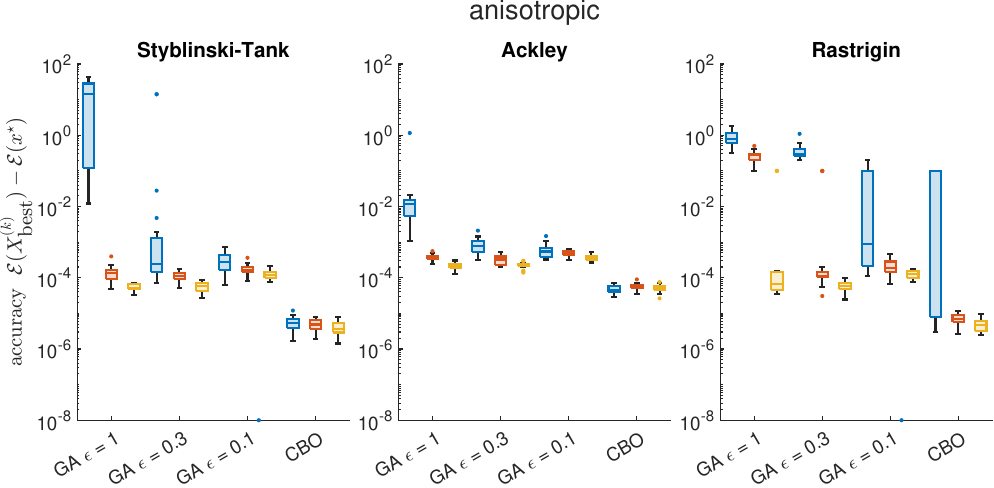}
\caption{Anisotropic diffusion ($D = x_* - x$) without cooling strategy $\sigma_{(k)} = \sigma_{(0)}$}
\label{fig:scaling:accuracy:b}
\end{subfigure}

\caption{Accuracy $\E(X^{(k)}_\textup{best}) - \E(x^\star)$ at the last iteration $k = k_{\max} = 300$ for algorithms \eqref{eq:GA2CBO} ($\ve = 1$), \eqref{eq:GA2CBOeps} ($\ve  =0.3$), \eqref{eq:GA2CBOtau} ($\ve  =\tau  = 0.1$) and CBO. Problems' dimension is $\d =10$. Parameters are set to $\tau = 0.1, \alpha = 10^4, \lambda = 1, \sigma_{(0)} = 1$. the initial particles' locations are independently sampled from $\textup{Unif}[-2,2]^\d$. The box charts display median, lower and upper quartiles, and outliers (computed with interquartile range), for the 100 experiments performed. }

\label{fig:scaling:accuracy}
\end{figure}

\section{Conclusions}
\label{sec:conclusions}

In this work, we have presented new mathematical approach for the theoretical analysis of
Genetic Algorithms (GA) inspired by the kinetic theory for colliding particles in statistical physics.
The individuals characterizing the GA iteration are modeled as random variables interacting ad
discrete times, and their dynamics is approximated by assuming the propagation of chaos property
for large population sizes. We further derived a continuous-in-time kinetic model consisting of an
equation of Boltzmann type and established relations with other kinetic and mean-field models in
optimization, such as kinetic-based optimization (KBO) and consensus-based optimization (CBO).

 The proposed kinetic description can be applied both to GA algorithms with fitness-based and
rank-based selection mechanisms and allows for a rigorous convergence analysis of such metaheuristic
optimization methods. For GA with Boltzmann selection, in particular, we have been able to prove
that the particle mean reaches an arbitrarily small neighborhood of the global solution under
mild assumptions on the objective function. This was done by exploiting the quantitative Laplace
principle proved in \cite{fornasier2021consensusbased}. The performed numerical experiments validated the propagation of chaos
assumption and showed the steady states reached by the GA dynamics for different parameter
choices. Benchmark tests in dimension $\d = 10$ showed the performance of the GA algorithms for
different scalings of the parameters.
 
 The results presented here represent a further step forward in the understanding of metaheuristic
algorithms and their mathematically rigorous formulation through the lens of statistical physics.
These results stand as a natural continuation and completion of similar results for simulated
annealing, particle swarm optimization and consensus based optimization \cite{benfenati2021binary,albi2023kinetic,pareschi23,Grassi21,carrillo2018analytical,Hui22}. 

In the sequel, we plan to address our attention to the convergence analysis of rank-based dynamics
and to prove the propagation of chaos property of GA particle systems. A particularly challenging
aspect is to derive a quantitative estimate of the error of the kinetic approximation in terms of
the number of $N$ particles. This will allow to extend the convergence analysis performed at the
kinetic level directly to the GA particle system, providing a complete mathematical framework for
the analysis of the GA heuristic.

\section*{Acknowledgment}
This work has been written within the activities of GNCS group of INdAM (National Institute of High Mathematics) and is partially supported by ICSC – Centro Nazionale di Ricerca in High Performance Computing, Big Data and Quantum Computing, funded by European Union – NextGenerationEU. L.P. acknowledges the support of the Royal Society Wolfson Fellowship 2023-2028 “Uncertainty quantification, data-driven simulations and learning of multiscale complex systems governed by PDEs”.  The work of G.B. is funded by the Deutsche Forschungsgemeinschaft (DFG, German Research Foundation) through 320021702/GRK2326 ``Energy, Entropy, and Dissipative Dynamics (EDDy)''.

\bibliographystyle{abbrv}
\bibliography{bibfile}


\end{document}